\documentclass{amsart}
\usepackage{amsfonts}
\usepackage{graphicx}
\usepackage{amsmath}
\usepackage{amsthm}
\usepackage{amssymb}
\usepackage{enumerate}
\usepackage[all, cmtip]{xy} 
\newtheorem{theorem}{Theorem}

\newtheorem{corollary}[theorem]{Corollary}

\newtheorem*{mydef}{Definition}
\newtheorem*{myexample}{Example}

\newtheorem{lemma}[theorem]{Lemma}

\newtheorem{problem}[theorem]{Problem}
\newtheorem{proposition}[theorem]{Proposition}
\newtheorem{remark}[theorem]{Remark}

\newcommand{\Ker}{\operatorname{Ker}}
\newcommand{\Image}{\operatorname{Image}}
\newcommand{\End}{\operatorname{End}}
\newcommand{\Hom}{\operatorname{Hom}}

\begin{document}
\title{Dual Automorphism-invariant Modules}
\author{Surjeet Singh}
\address{House No. 424, Sector No. 35 A, Chandigarh-160036, India}
\email{ossinghpal@yahoo.co.in}
\author{Ashish K. Srivastava}
\address{Department of Mathematics and Computer Science, St. Louis University, St.
Louis, MO-63103, USA}
\email{asrivas3@slu.edu}
\keywords{discrete modules, lifting modules, perfect ring, pseudo-projective modules, quasi-projective modules.}
\subjclass[2000]{16D50, 16P40}

\begin{abstract}
A module $M$ is called an automorphism-invariant module if every isomorphism between two essential submodules of $M$ extends to an automorphism of $M$. This paper introduces the notion of dual of such modules. We call a module $M$ to be a dual automorphism-invariant module if whenever $K_1$ and $K_2$ are small submodules of $M$, then any epimorphism $\eta:M/K_1\rightarrow M/K_2$ with small kernel lifts to an endomorphism $\varphi$ of $M$. In this paper we give various examples of dual automorphism-invariant module and study its properties. In particular, we study abelian groups and prove that dual automorphism-invariant abelian groups must be reduced. It is shown that over a right perfect ring $R$, a lifting right $R$-module $M$ is dual automorphism-invariant if and only if $M$ is quasi-projective.
\end{abstract}

\maketitle

\bigskip

\noindent All our rings have identity element and modules are right unital. A right $R$-module $M$ is called an {\it automorphism-invariant module} if every isomorphism between two essential submodules of $M$ extends to an automorphism of $M$. Equivalently, $M$ is an automorphism-invariant module if for any automorphism $\sigma$ of $E(M)$, $\sigma(M)\subseteq M$ where $E(M)$ is the injective hull of $M$ (see \cite{ZL} and \cite{SS}). 

Recall that a right $R$-module $M$ is called a {\it quasi-injective module} ({\it pseudo-injective module}) if $M$ is invariant under any endomorphism (monomorphism) of $E(M)$. Thus, clearly, any quasi-injective module or pseudo-injective module is automorphism-invariant. 
In this paper we introduce the notion of dual of an automorphism-invariant module. 

 A submodule $N$ of a module $M$ is called {\it small} in $M$ (denoted as $N\subset_{s} M$) if $N+K\neq M$ for any proper submodule $K$ of $M$. The Jacobson radical of a module $M$ is the sum of all small submodules of $M$ and is denoted by $J(M)$. For any term not defined here the reader is referred to \cite{CLVW} and \cite{MM}.  

\begin{mydef}
A right $R$-module $M$ is called a dual automorphism-invariant module if whenever $K_1$ and $K_2$ are small submodules of $M$, then any epimorphism $\eta:M/K_1\rightarrow M/K_2$ with small kernel lifts to an endomorphism $\varphi$ of $M$.
\end{mydef}

\bigskip
\[
\xymatrix{
M \ar[d]^{}  \ar[r]^{\varphi} &M \ar[d]^{}\\
M/K_1 \ar[r]^{\eta}                 &M/K_2}
\]

\bigskip

\noindent We will show that, in fact, the above endomorphism $\varphi$ must be an automorphism of $M$. First, we have the following 

\begin{lemma} \label{insert1}
Let $M$ be a dual automorphism-invariant module. If $\varphi:M\rightarrow M$ is an epimorphism with small kernel, then $\varphi$ is an automorphism. 
\end{lemma}

\begin{proof}
Let $K=\Ker(\varphi)$. Then $\varphi$ induces an isomorphism $\bar{\varphi}: \frac{M}{K}\rightarrow M$. Consider $\bar{\varphi}^{-1}:M\rightarrow \frac{M}{K}$. Since $M$ is a dual automorphism-invariant module, by definition, $\bar{\varphi}^{-1}$ lifts to an endomorphism $\lambda:M\rightarrow M$. We have $\lambda(M)+K=M$. As $K\subset_{s} M$, we get $\lambda(M)=M$. Thus $\lambda$ is an epimorphism. Then for any $x\in M$, $\bar{\varphi}^{-1}(x)=\lambda(x)+K$. Now $x=\bar{\varphi}\bar{\varphi}^{-1}(x)=\bar{\varphi}(\lambda(x)+K)=\varphi \lambda(x)$. This proves that $\varphi \lambda=1_{M}$. Thus $\varphi^{-1}=\lambda$ and hence $\varphi$ is an automorphism.    
\end{proof}

\noindent As a consequence, it follows that 

\begin{corollary} \label{insert2}
A right $R$-module $M$ is a dual automorphism-invariant module if and only if for any two small submodules $K_1$ and $K_2$ of $M$, any epimorphism $\eta:M/K_1\rightarrow M/K_2$ with small kernel lifts to an automorphism $\varphi$ of $M$.
\end{corollary}

\begin{proof}
Let $M$ be a dual automorphism-invariant right $R$-module. Let $K_1$ and $K_2$ be any two small submodules of $M$ and let $\eta:M/K_1\rightarrow M/K_2$ be any epimorphism with small kernel. Let $\ker(\eta)=L/K_1$. Then $L$ is small in $M$. If $\pi:M\rightarrow M/K_1$ is a canonical epimorphism, then $\lambda=\eta \pi:M\rightarrow M/K_2$ has kernel $L$. Thus $\lambda:M\rightarrow M/K_2$ is an epimorphism with small kernel. By definition, $\lambda$ lifts to an endomorphism $\varphi$ of $M$. Now $\varphi(M)+K_2=M$. As $K_2\subset_{s}M$, we get $\varphi(M)=M$. Thus $\varphi$ is an epimorphism with small kernel, and hence by above lemma, $\varphi$ is an automorphism. 
The converse is obvious.     
\end{proof}

\bigskip

\begin{myexample}
A module with no non-zero small submodule is easily seen to be a dual automorphism-invariant module. Thus all the semiprimitive modules belong to the family of dual automorphism-invariant modules. In particular, the regular modules studied by Zelmanowitz in \cite{Z} are dual automorphism-invariant. 
\end{myexample}

\bigskip

\section{$V$-rings and Dual automorphism-invariant modules}

\bigskip

\noindent Recall that a ring $R$ is called a right $V$-ring if every simple right $R$-module is injective. The class of right $V$-rings was introduced by Villamayor \cite{MV}. It is a well-known unpublished result due to Kaplansky that a commutative ring is von Neumann regular if and only if it is a $V$-ring. The class of $V$-rings includes von Neumann regular rings with artinian primitive factors. It is well-known that if $R$ is a right $V$-ring then for every right $R$-module $M$, $J(M)=0$ and so $M$ has no nonzero small submodule. For the sake of completeness, we present the proof in the next proposition.   

\begin{proposition} \label{vring}
Let $R$ be a right $V$-ring. Then every right $R$-module is dual automorphism-invariant.
\end{proposition} 

\begin{proof}
Let $M$ be a nonzero right $R$-module. Let $x \,(\neq 0) \in M$. By Zorn's lemma there exists a submodule $N$ of $M$ maximal with respect to not containing $x$. Then the intersection of all nonzero submodules of $M/N$ is $(xR+N)/N$ and it is simple. Since $R$ is a right $V$-ring, $(xR+N)/N$ is injective. Then $(xR+N)/N$ being a summand of $M/N$ gives $M/N=(xR+N)/N$. Thus $M=xR+N$. This shows that $M$ has no nonzero small submodule and consequently, $M$ is dual automorphism-invariant. 
\end{proof}

\noindent It is quite natural to ask here whether the converse of above result also holds. We proceed to answer this in the affirmative but first, we have the following useful observation. 

\begin{lemma} \label{dsdual}
Let $M_1, M_2$ be right $R$-modules. If $M=M_1\oplus M_2$ is dual automorphism-invariant, then any homomorphism $f:M_1\rightarrow M_2/K_2$ with $K_2$ small in $M_2$ and $\Ker(f)$ small in $M_1$ lifts to a homomorphism $g:M_1\rightarrow M_2$.
\end{lemma} 

\begin{proof}
We have an epimorphism $\sigma:M\rightarrow \frac{M}{K_2}$ given by $\sigma(m_1+m_2)=m_1+f(m_1)+(m_2+K_2)$ for $m_1\in M_1, m_2\in M_2$. Since $K_2$ is small in $M_2$ and $M_2\subset M$, we get that $K_2$ is small in $M$. Now, as $M$ is dual automorphism-invariant, by Corollary \ref{insert2}, $\sigma$ lifts to an automorphism $\eta$ of $M$. Let $x_1\in M_1$ and $\eta(x_1)=u_1+u_2$ where $u_1\in M_1, u_2\in M_2$. Then $u_1+u_2+K_2=(x_1+K_2)+f(x_1)$, which gives $u_2+K_2=f(x_1)$. Let $\pi_2: M\rightarrow M_2$ be the natural projection. Then $g=\pi_2 \eta|_{M_1}: M_1\rightarrow M_2$ lifts $f$.   
\end{proof} 

Now we are ready to prove the following characterization of right $V$-rings in terms of dual automorphism-invariant modules.

\begin{theorem} \label{vringchar}
A ring $R$ is a right $V$-ring if and only if every finitely generated right $R$-module is dual automorphism-invariant.
\end{theorem}

\begin{proof}
Suppose every finitely generated right $R$-module is dual automorphism-invariant. We wish to show that $R$ is a right $V$-ring. Assume to the contrary that $R$ is not a right $V$-ring. Then there exists a simple right $R$-module $S$ such that $S$ is not injective. Let $E(S)$ be the injective hull of $S$. Then $E(S)\neq S$. Choose any $x\in E(S)\setminus S$. Then $S$ is small in $xR$ and $xR$ is uniform. Let $A=ann_r(x)$. As $S$ is a submodule of $xR\cong R/A$, we may take $S=B/A$ for some $A\subset B\subset R_R$. Consider $M=\frac{R}{A}\times \frac{R}{B}$. As $M$ is finitely generated, by hypothesis $M$ is dual automorphism-invariant. We have the identity homomorphism $1_{R/B}: R/B \rightarrow R/B\cong \frac{R/A}{B/A}$ where $\Ker(1_{R/B})=0$ is small in $R/B$ and $B/A$ is small in $R/A$. By Lemma \ref{dsdual}, the identity mapping on $R/B$ can be lifted to a homomorphism $\eta: \frac{R}{B} \rightarrow \frac{R}{A}$. Thus $\Image(\eta)$ is a summand of $R/A$, which is a contradiction to the fact that $R/A\,(\cong xR)$ is uniform. Hence $R$ is a right $V$-ring.

The converse is obvious from the Proposition \ref{vring}.
\end{proof}

\begin{remark}
It may be noted here that if we weaken the hypothesis above and assume that $R$ is a ring such that every cyclic right $R$-module is dual automorphism-invariant, then $R$ need not be a right $V$-ring. We know that every cyclic module over a commutative ring is quasi-projective and it will be shown in Corollary \ref{pseudo} that every pseudo-projective and hence quasi-projective module is dual automorphism-invariant. Thus, if we consider $R$ to be a commutative ring which is not von-Neumann regular, then every cyclic module over $R$ is dual automorphism-invariant but $R$ is not a $V$-ring.  
\end{remark}


\bigskip

\section{More Examples of Dual automorphism-invariant modules}

\noindent In this section we will discuss various other examples of dual automorphism-invariant modules. A module $M$ is called a {\it quasi-projective module} ({\it pseudo-projective module}) if for every submodule $N$ of $M$, any homomorphism (epimorphism) $\varphi: M\rightarrow M/N$ can be lifted to a homomorphism $\psi: M\rightarrow M$, that is, the diagram below commutes.   

\bigskip
\[
\xymatrix{
&M \ar[ld]_{\psi}  \ar[d]^{\varphi}\\
M \ar[r]^{\pi}                 &M/N}
\]
\bigskip

\noindent Clearly, every quasi-projective module is pseudo-projective.

\begin{proposition} \label{pseudo}
Any pseudo-projective module is dual automorphism-invariant.
\end{proposition}

\begin{proof}
Suppose $M$ is a pseduo-projective module. Let $L_1, L_2$ be two small submodules of $M$ and $\sigma:\frac{M}{L_1}\rightarrow \frac{M}{L_2}$ be an epimorphism. Let $\pi_1: M\rightarrow \frac{M}{L_1}$ be a natural mapping. As $M$ is pseudo-projective, $\sigma \pi_1$ lifts to an endomorphism $\eta$ of $M$. Let $\pi_2: M\rightarrow \frac{M}{L_2}$ be a natural mapping. Then $\pi_2 \eta=\sigma \pi_1$. Therefore $\pi_2 \eta (L_1)=\sigma \pi_1(L_1)=0$ gives $\eta(L_1)\subseteq L_2$. Hence $\eta$ is a lifting of $\sigma$. This proves that $M$ is dual automorphism-invariant.   
\end{proof}

\noindent Now we will show that dual automorphism-invariant modules need not be pseudo-projective. But, first we have the following useful observation.

\begin{lemma} \label{ds}
Let $M_1, M_2$ be right $R$-modules. If $M=M_1\oplus M_2$ is pseudo-projective, then $M_1$ is $M_2$-projective and $M_2$ is $M_1$-projective. 
\end{lemma} 

\begin{proof}
Let $f: M_1\rightarrow M_2/N$ be a homomorphism. It induces an epimorphism $\sigma: M\rightarrow M/N$ given by $\sigma(x_1+x_2)=x_1+f(x_1)+ (x_2+N)$ for $x_1\in M_1, x_2\in M_2$. Since $M$ is pseudo-projective, $\sigma$ lifts to an endomorphism $\eta$ of $M$. Let $x_1\in M_1$ and $\eta(x_1)=u_1+u_2$ where $u_1\in M_1, u_2\in M_2$. Then $u_1+u_2+N=x_1+f(x_2)\in M_1\oplus \frac{M_2}{N}$, $u_2+N=f(x_2)$. 

Let $\pi_2: M\rightarrow M_2$ be the natural projection. Then $\pi_2 \eta |_{M_1}: M_1\rightarrow M_2$ is such that $\pi_2 \eta(x_1)=u_2$. This shows that $g=\pi_2 \eta |_{M_1}: M_1\rightarrow M_2$ lifts $f$. Hence $M_1$ is $M_2$-projective. Similarly it can be shown that $M_2$ is $M_1$-projective.  
\end{proof} 

\begin{proposition} \label{ss}
If every right module over a ring $R$ is pseudo-projective, then $R$ is semisimple artinian.
\end{proposition}

\begin{proof}
Let $A$ be any right ideal of $R$. Since every right $R$-module is pseudo-projective, $R\oplus \frac{R}{A}$ is pseudo-projective. By Lemma \ref{ds}, $R/A$ is $R$-projective. Therefore the identity mapping on $R/A$ lifts to a mapping from $R/A$ to $R$. Thus the exact sequence $0\rightarrow A\rightarrow R\rightarrow R/A\rightarrow 0$ splits. Therefore $A$ is a summand of $R$. This shows that every right ideal of $R$ is a summand of $R$. Hence $R$ is semisimple artinian. 
\end{proof}

\begin{remark}
If $R$ is a right $V$-ring which is not right artinian (for example, a non-artinian commutative von Neumann regular ring), then by Proposition \ref{vring} and Proposition \ref{ss}, it follows that $R$ admits a dual automorphism-invariant module which is not pseudo-projective. 
\end{remark}

\bigskip

\section{Properties of dual automorphism-invariant modules}

\noindent In this section we discuss various properties of dual automorphism-invariant modules.

\begin{proposition} \label{summand}
Any direct summand of a dual automorphism-invariant module is dual automorphism-invariant.
\end{proposition}

\begin{proof}
Let $M$ be a dual automorphism-invariant right $R$-module and let $M=A\oplus B$. Let $K_1, K_2$ be two small submodules of $A$ and $\sigma: \frac{A}{K_1}\rightarrow \frac{A}{K_2}$ be an epimorphism with $\Ker(\sigma) \subset_s \frac{A}{K_1}$. Clearly, $K_1, K_2$ are small in $M$ and $\sigma'=\sigma \oplus 1_{B}: \frac{M}{K_1}  \rightarrow \frac{M}{K_2}$ is an epimorphism with $\Ker(\sigma')\subset_s \frac{M}{K_1}$. Since $M$ is dual automorphism-invariant, $\sigma'$ lifts to an endomorphism $\eta$ of $M$. For the inclusion map $i_1:A\rightarrow M$ and the projection $\pi_1:M\rightarrow A$, the map $\pi_1 \eta i_1: A\rightarrow A$ lifts $\sigma$. Hence $A$ is dual automorphism-invariant. This shows that any direct summand of a dual automorphism-invariant module is dual automorphism-invariant. 
\end{proof}

\begin{remark}
\begin{enumerate}[(i)]
\item The direct sum of two dual automorphism-invariant modules need not be dual automorphism-invariant. For example, $\mathbb Z_2$ and $\mathbb Z_4$ are dual automorphism-invariant $\mathbb Z$-modules but $\mathbb Z_2 \oplus \mathbb Z_4$ is not a dual automorphism-invariant $\mathbb Z$-module. 

\bigskip

\item The submodules of a dual automorphism-invariant module need not be dual automorphism-invariant. For example, $M=\frac{\mathbb Z}{8\mathbb Z}\oplus \frac{\mathbb Z}{8\mathbb Z}$ is a dual automorphism-invariant $\mathbb Z$-module but $N=\frac{2\mathbb Z}{8\mathbb Z}\oplus \frac{\mathbb Z}{8\mathbb Z} \subset M$ is not dual automorphism-invariant.
\end{enumerate}
\end{remark}

\bigskip

\noindent A module $M$ is called a {\it hollow} module if every proper submodule of $M$ is small in $M$. A module is called {\it local} if it is hollow and has a unique maximal submodule. 

\bigskip

\noindent For the direct sum of local modules, we have the following

\begin{proposition}
If $M_1, M_2$ are two local modules such that $M_1\oplus M_2$ is dual automorphism-invariant, then $M_1$ is $M_2$-projective and $M_2$ is $M_1$-projective.
\end{proposition}

\begin{proof}
Consider any diagram 
\bigskip
\[
\xymatrix{
&M_1 
\ar[d]^{f}\\
M_2 \ar[r]^{g}                 &M_2/K\ar[r]   &0}
\]
with exact row. Since $M_1$ and $M_2$ are local, $K$ is a small submodule of $M_2$ and $\Ker(f)$ is a small submodule of $M_1$. Therefore, by Lemma \ref{dsdual}, $f$ lifts to a homomorphism $h:M_1\rightarrow M_2$. This shows that $M_1$ is $M_2$-projective. Similarly it can be shown that $M_2$ is $M_1$-projective.   
\end{proof} 

\noindent Consider the following conditions on a module $N$;

\bigskip

\noindent ($D1$): For every submodule $A$ of $N$, there exists a decomposition $N=N_1\oplus N_2$ such that $N_1\subseteq A$ and $N_2\cap A \subset_{s} N$.

\bigskip

\noindent ($D2$): If $A$ is a submodule of $N$ such that $N/A$ is isomorphic to a direct summand of $N$, then $A$ is a direct summand of $N$.

\bigskip

\noindent ($D3$): If $A$ and $B$ are direct summands of $N$ with $A+B=N$, then $A\cap B$ is a direct summand of $N$.  

\bigskip

\noindent It is well-known that if a module $N$ satisfies the condition ($D2$), then it also satisfies the condition ($D3$). If $N$ satisfies the condition ($D1$), then it is called a {\it lifting module}. If $N$ satisfies the conditions ($D1$) and ($D3$), then it is called a {\it quasi-discrete module}. If $N$ satisfies the conditions ($D1$) and ($D2$), then it is called a {\it discrete module}. The following implication is well-known;
\begin{center}
discrete $\implies$ quasi-discrete $\implies$ lifting 
\end{center} 

\noindent Since any quasi-projective module satisfies the property ($D2$) and hence the property ($D3$), it is natural to ask whether a dual automorphism-invariant module satisfies the property ($D2$). We do not know the answer to this question, however we are able to show in the next proposition that every supplemented dual automorphism-invariant module satisfies the property ($D3$). Recall that a submodule $K$ is called a {\it supplement} of $N$ in $M$ if $K$ is minimal with respect to the property that $K+N=M$. As a consequence, it follows that $K\cap N$ is small in $K$ and hence in $M$. A module $M$ is called a {\it supplemented module} if every submodule of $M$ has a supplement.   

\begin{proposition} \label{d3}
If $M$ is a supplemented dual automorphism-invariant module, then $M$ satisfies the property ($D3$).
\end{proposition}

\begin{proof}
Let $M$ be a supplemented dual automorphism-invariant module. Let $A$ and $B$ be direct summands of $M$ such that $A+B=M$. We wish to show that $A\cap B$ is a direct summand of $M$. Since $M$ is a supplemented module, there exists a submodule $C$ of $M$ such that $A\cap B+C=M$ and $A\cap B\cap C\subset_{s} M$. Now, clearly we have $B=A\cap B+B\cap C$ and $A=A\cap B+A\cap C$. This gives $M=A\cap B+B\cap C+ A\cap C$. Set $L=A\cap B\cap C$. 

Now, as $C=A\cap C+B\cap C$, we have $L=A\cap B \cap (A\cap C + B\cap C)\subset_{s} M$. Thus, 
\[\frac{M}{L}=\frac{A\cap B}{L} \oplus \frac{A\cap C}{L} \oplus \frac{B\cap C}{L}.\]
 Since $A$ is a direct summand of $M$, we have $M=A\oplus A'$ for some submodule $A'$ of $M$. Then 
\[\frac{M}{L}=\frac{A}{L}\oplus \frac{A'+L}{L} = \frac{A\cap B}{L} \oplus \frac{A\cap C}{L} \oplus \frac{A'+L}{L}.\]
 Set $T=\frac{A\cap B}{L} \oplus \frac{A'+L}{L}$. Let $\pi: M/L \rightarrow T$ be the natural projection. Let us denote the restriction of $\pi$ to $T$ by $\pi_{T}$. Then $\pi_{T}: T\rightarrow T$ is an isomorphism. Thus we have an isomorphism \[1_{A\cap C/L} \oplus \pi_{T}: M/L \rightarrow M/L.\]
 
 Since $M$ is dual automorphism-invariant, this map lifts to an automorphism 
 \[\eta: M\rightarrow M.\] 
We have 
 \[\eta(B)=(A\cap B) + (A'+L)=(A\cap B)+A'=(A\cap B)\oplus A'.\] 
 This shows that $A\cap B$ is a direct summand of $\eta(B)$. Now as $\eta(B)$ is a direct summand of $M$, we have that $A\cap B$ is a direct summand of $M$. Thus $M$ satisfies the property $(D3)$. 
\end{proof}

\bigskip

\section{Dual automorphism-invariant abelian groups}

\noindent In this section we study dual automorphism-invariant abelian groups. We begin with the following useful result which will help us in constructing more examples of dual automorphism-invariant modules.
\begin{proposition}
Let $P$ be a projective right $R$-module that has no nonzero small submodule, and $M$ be any quasi-projective right $R$-module such that $\Hom_R(\frac{M}{K}, P)=0$ for any small submodule $K$ of $M$. Then $P\oplus M$ is dual automorphism-invariant. 
\end{proposition}

\begin{proof}
Set $N=P\oplus M$. We have projections $\pi_1:N\rightarrow P$, and $\pi_2:N\rightarrow M$. Let $K$ be a small submodule of $N$. Then $\pi_1(K)\subset_s P$. This gives $\pi_1(K)=0$ as $P$ has no nonzero small submodule. Therefore $K\subset M$. Let $K_1, K_2$ be two small submodules of $N$. Then 
\[\frac{N}{K_1}=P\oplus \frac{M}{K_1}, \quad \frac{N}{K_2}=P\oplus \frac{M}{K_2}.\] 

\noindent Let $\sigma: \frac{N}{K_1}\rightarrow \frac{N}{K_2}$ be an epimorphism. Now $\sigma$ may be viewed as $\sigma= \left[ 
\begin{array}{cc}
\sigma_{11} & \sigma_{12} \\ 
\sigma_{21} & \sigma_{22}
\end{array}
\right] $, where $\sigma_{11}:P\rightarrow P$, $\sigma_{12}:\frac{M}{K_1}\rightarrow P$, $\sigma_{21}:P\rightarrow \frac{M}{K_2}$, $\sigma_{22}:\frac{M}{K_1}\rightarrow \frac{M}{K_2}$. 
Set $\lambda_{11}=\sigma_{11}$, $\lambda_{12}:M\rightarrow P$ naturally given by $\sigma_{12}$, and $\lambda_{21}:P\rightarrow M$ a lifting of $\sigma_{21}$. As $M$ is quasi-projective, $\sigma_{22}$ lifts to an endomorphism $\lambda_{22}$ of $M$. 

\noindent Let $\lambda=\left[ 
\begin{array}{cc}
\lambda_{11} & \lambda_{12} \\ 
\lambda_{21} & \lambda_{22}
\end{array}
\right] $. Then $\lambda$ is an endomorphism of $N$. As $\lambda_{12}=0$ by the hypothesis, for any $x\in K_1$, \[ \left[ 
\begin{array}{cc}
\lambda_{11} & \lambda_{12} \\ 
\lambda_{21} & \lambda_{22}
\end{array}
\right]  \left[ 
\begin{array}{c}
0 \\ 
x
\end{array}
\right] = \left[ 
\begin{array}{c}
\lambda_{12}(x) \\ 
\lambda_{22}(x)
\end{array}
\right] = \left[ 
\begin{array}{c}
0 \\ 
\lambda_{22}(x)
\end{array}
\right]. \] As $\lambda_{22}$ is a lifting of $\sigma_{22}$, $\lambda_{22}(K_1)\subseteq K_2$. Hence $\lambda$ lifts $\sigma$. This proves that $P\oplus M$ is dual automorphism-invariant. 
\end{proof}

\noindent In particular, for abelian groups we have the following

\begin{corollary} \label{dualabelianexample}
Let $P$ be a projective abelian group and let $M$ be any torsion quasi-projective abelian group. Then $P\oplus M$ is dual automorphism-invariant.
\end{corollary}

\begin{proof}
As $P$ is a direct sum of copies of $\mathbb Z$ and $\mathbb Z$ has no nonzero small subgroup, $P$ has no nonzero small subgroup. For any small submodule $K$ of $M$, since $M/K$ is a torsion abelian group, $\Hom_{\mathbb Z}(\frac{M}{K}, \mathbb Z)=0$ and hence $\Hom_{\mathbb Z}(\frac{M}{K}, P)=0$. Thus the result follows from the above lemma. 
\end{proof}

\noindent The above result gives us plenty of examples of dual automorphism-invariant modules. 
\begin{myexample}
Let $M=\mathbb Z \oplus C$, where $C$ is a finite cyclic group. By Corollary \ref{dualabelianexample}, $M$ is dual automorphism-invariant but $M$ is not pseudo-projective unless $C=0$. 
\end{myexample}

\noindent We recall here some useful facts about abelian groups. For details we refer the reader to Fuchs \cite{Fuchs}. Let $G$ be an abelian group. An element $x\in G$ is said to be of {\it finite height}, if there exists an upper bound on all positive integers $k$ such that $p^ky=x$ for some prime number $p$ and some $y\in G$. An abelian group is said to be {\it bounded}, if there is an upper bound on the orders of its elements. A bounded abelian group is a direct sum of cyclic groups \cite[Theorem 17.2]{Fuchs}. A subgroup $H$ of $G$ is said to be {\it pure} in $G$, if $nG\cap H=nH$ for every integer $n$. If an element $x\in G$ is of order a prime number $p$ and has finite height, then there exists a summand $H$ of $G$ of finite order such that $x\in H$ \cite[Corollary 27.2]{Fuchs}. If a pure subgroup $H$ of $G$ is bounded, then $H$ is a summand of $G$ \cite[Theorem 27.5]{Fuchs}. The torsion subgroup of an abelian group is a pure subgroup. It follows that if the torsion subgroup $T$ of $G$ is bounded, then $T$ is a summand of $G$. An abelian group $G$ is called a {\it divisible group} if for each positive integer $n$ and every element $g\in G$, there exists $h\in G$ such that $nh=g$. An abelian group $G$ is called a {\it reduced group} if $G$ has no proper divisible subgroup. 

\begin{theorem} \label{divisible}
(see \cite{Fuchs}) If $G$ is an abelian group, then $G=D\oplus K$, where $D$ is divisible and $K$ is reduced. Furthermore, the structure of divisible abelian group is given as 
\[D\cong (\oplus_{m_p} \mathbb Z(p^{\infty}))\oplus (\oplus_{n} \mathbb Q).\]
\end{theorem}

We have the following observation for a torsion abelian group.

\begin{lemma} \label{torsiondivisible}
Let $G$ be a torsion abelian group such that $G$ is dual automorphism-invariant. Then $G$ is reduced. 
\end{lemma}

\begin{proof}
Assume to the contrary that $G$ is not reduced. Then in view of Theorem \ref{divisible}, we have $G\cong \oplus_{m_p} \mathbb Z(p^{\infty})$. For a prime number $p$, consider $H=\mathbb Z(p^{\infty})$. Its every proper subgroup is small. Let $A\subsetneq B$ be two proper subgroups of $H$. There exists an isomorphism $\sigma: \frac{H}{A}\rightarrow \frac{H}{B}$. Since every summand of a dual automorphism-invariant module is dual automorphism-invariant, $H$ is dual automorphism-invariant. Therefore $\sigma$ lifts to an endomorphism $\eta$ of $H$. Then $\sigma(A)=B$. This gives a contradiction as order of $A$ is less than the order of $B$. Hence $G$ is reduced.
\end{proof}

Next, we recall the characterization of quasi-projective abelian groups due to Fuchs and Rangaswamy \cite{FR}.

\begin{theorem} \label{qpabelian}
(Fuchs and Rangaswamy \cite{FR}) An abelian group $G$ is quasi-projective if and only if it is either free or a torsion group such that every $p$-component $G_p$ is a direct sum of cyclic groups of the same order $p^n$.
\end{theorem}

Now we are ready to prove the following for a torsion abelian group.

\begin{theorem}
Let $G$ be a torsion abelian group. Then the following are equivalent;
\begin{enumerate}[(i)]
\item $G$ is dual automorphism-invariant.
\item $G$ is quasi-projective.
\item $G$ is discrete.
\end{enumerate} 
\end{theorem}

\begin{proof}
(i)$\implies$(ii). Since any abelian group is a direct sum of a divisible group and a reduced group, in view of Lemma \ref{torsiondivisible}, it follows that $G$ is reduced. Let $p$ be a prime number. Consider the $p$-component $G_p$ of $G$. Suppose $G_p\neq 0$. As $G_p$ is reduced, $G_p=A_1\oplus L$, where $A_1$ is a nonzero cyclic $p$-group. Now $o(A_1)=p^n$ for some $n>0$. If $L=0$, we get that $G_p$ is quasi-projective. Suppose $L\neq 0$. Then $L=A_2\oplus L_1$, where $A_2$ is a nonzero cyclic $p$-group. By Proposition \ref{summand}, $A_1\oplus A_2$ is dual automorphism-invariant. As every subgroup of $A_1$ or $A_2$ is small, it follows that $A_1$ is $A_2$-projective and $A_2$ is $A_1$-projective. Hence $A_1\oplus A_2$ is quasi-projective. This gives $A_1\cong A_2$. By above theorem, we get $G_p$ is a direct sum of copies of $A_1$. Hence $G_p$ is quasi-projective. This proves that $G$ itself is quasi-projective.   

\bigskip

\noindent (ii)$\implies$(i). This follows from Proposition \ref{pseudo}.

\bigskip

\noindent This shows that (i) and (ii) are equivalent. For the equivalence of (ii) and (iii), see \cite[Theorem 5.5]{MS}. 
\end{proof}

\begin{lemma} \label{torsionfree}
Let $G$ be a torsion-free, uniform abelian group which is not finitely generated. Let $H$ be a nontrivial cyclic subgroup of $G$. For any prime number $p$, let $G_p=\{x\in G: p^nx\in H$ for some $n\ge 0\}$. Then $J(G)\neq 0$ if and only if the number of prime numbers $p$ for which $G_p=H$ is finite. 
\end{lemma}

\begin{proof}
Observe that $H\subseteq G_p$ for any prime number $p$. Without loss of generality we take $G\subseteq \mathbb Q$ and $H=\mathbb Z$. Let $M$ be a maximal subgroup of $G$. For some prime number $p$, $G/M$ is of order $p$. Thus $pG_p\subseteq M$. Now $G_p$ is generated by some powers $\frac{1}{p^n}$, $n\ge 0$.  

\bigskip

\noindent Case $1$. Assume $\mathbb Z \subset G_p$. Then $\mathbb Z\subseteq pG_p \subseteq M$, $M/{\mathbb Z}$ is a maximal subgroup of $G/{\mathbb Z}$. As $G/{\mathbb Z}$ is a torsion group such that for each prime number $q$, $G_{q}/{\mathbb Z}$ is the $q$-torsion component of $G/{\mathbb Z}$, we get $G_q\subseteq M$, whenever $q\neq p$. Then $M=(G_p\cap M) + A_p$, where $A_p$ is the sum of all $G_q$, $q\neq p$.

\bigskip

\noindent Case $2$. Assume $\mathbb Z=G_p$. If $\mathbb Z\subseteq M$, the arguments of Case $1$ show that $M=G$, which is a contradiction. Thus $\mathbb Z\not\subseteq M$, and we get $M\cap \mathbb Z=p \mathbb Z$. 

\bigskip

\noindent We know that the intersection of infinitely many sets $p\mathbb Z$ is zero. Thus it follows that $J(G)\neq 0$ if and only if the number of primes $p$ for which $G_p=\mathbb Z$ is finite.
\end{proof}

\begin{theorem} \label{rad}
Let $G$ be a subgroup of $\mathbb Q$ containing $\mathbb Z$. Then the following conditions are equivalent;
\begin{enumerate}[(i)]
\item $G$ is dual automorphism-invariant.
\item The number of primes $p$ for which $G_p=\{x\in G: p^nx\in \mathbb Z\}=\mathbb Z$ is not finite.
\item $J(G)=0$.
\end{enumerate}
\end{theorem}

\begin{proof}
(i)$\implies$(ii). Let $G$ be a subgroup of $\mathbb Q$ containing $\mathbb Z$ and suppose $G$ is dual automorphism-invariant. Assume to the contrary that the number of primes $p$ for which $G_p=\{x\in G: p^nx\in \mathbb Z\}=\mathbb Z$ is not finite. Then by Lemma \ref{torsionfree}, $J(G)\neq 0$. Therefore we can find a cyclic subgroup $H$ that is small. We take $H=\mathbb Z$. By using the Lemma \ref{torsionfree}, we see that $G/{\mathbb Z}$ is an infinite direct sum of its $p$-components. For any prime number $p\neq 2$ for which the $p$-component $G_p/{\mathbb Z}$ is nonzero, its group of automorphisms is of order more than one. This proves that $Aut(G/\mathbb Z)$ is uncountable. As $\mathbb Q$ is countable, it follows that some automorphism of $G/{\mathbb Z}$ cannot be lifted to endomorphism of $G$. Hence $G$ is not dual automorphism-invariant, which is a contradiction. This proves that the number of primes $p$ for which $G_p=\{x\in G: p^nx\in \mathbb Z\}=\mathbb Z$ is not finite.

\bigskip

\noindent (ii)$\implies$(iii). It follows from Lemma \ref{torsionfree}.

\bigskip

\noindent (iii)$\implies$(i) is trivial.   
\end{proof}

\begin{corollary} \label{torsionfreedivisible}
If a torsion-free abelian group $G$ is dual automorphism-invariant, then it is reduced.
\end{corollary}

\begin{proof}
Let $G$ be a torsion-free dual automorphism-invariant abelian group. Assume that $G$ is not reduced. Then $G\cong \oplus_n \mathbb Q$. As $\mathbb Q$ is a summand of $G$, it must be dual automorphism-invariant by Proposition \ref{summand}. However, we know that $\mathbb Q$ is not dual automorphism-invariant (see Theorem \ref{rad}). This yields a contradiction. Hence $G$ is reduced.
\end{proof}

\noindent From Theorem \ref{divisible}, Lemma \ref{torsiondivisible} and Corollary \ref{torsionfreedivisible}, we conclude the following

\begin{theorem}
Let $G$ be a dual automorphism-invariant abelian group. Then $G$ must be reduced.
\end{theorem}

\bigskip

\section{Dual automorphism-invariant modules over right perfect rings} 

\noindent Bass \cite{Bass} defined a {\it projective cover} of a module $A$ to be an epimorphism $\mu:P\rightarrow A$ such that $P$ is a projective module and $\Ker(\mu)$ is a small submodule of $P$. Thus modules having projective covers are, up to isomorphism, of the form $P/K$, where $P$ is a projective module and $K$ is a small submodule of $P$. A ring $R$ is said to be a {\it right perfect ring} if every right $R$-module has a projective cover. 

Next, we proceed to provide an equivalent characterization for a module with projective cover to be a dual automorphism-invariant module. We begin with a lemma which will be used at several places throughout this paper.

\begin{lemma} \label{map1}
Let $A, B$ be right $R$-modules and let $C$ be a small submodule of $A$. Let $f:A\rightarrow B$, and $g:A\rightarrow B$ be homomorphisms such that $g(C)=0$. Consider induced homomorphisms $f':A\rightarrow B/f(C)$ and $g':A\rightarrow B/f(C)$. If $f'=g'$, then $f=g$.  
\end{lemma}

\begin{proof}
Let $\pi:B\rightarrow B/f(C)$ be the natural projection. Then $f'=\pi f$ and $g'=\pi g$. Now, since $f'=g'$, we have for each $x\in A$, $f'(x)=g'(x)$. Thus $\pi f(x)=\pi g(x)$ for each $x\in A$. This gives $f(x)+f(C)=g(x)+f(C)$ for each $x\in A$. So, $(f-g)(x)\in f(C)=(f-g)(C)$ for each $x\in A$. Therefore, $(f-g)(A)\subseteq (f-g)(C)$ and hence $A\subseteq C+\Ker(f-g)$. Now, since $C$ is a small submodule of $A$, we get $A=\Ker(f-g)$. Thus, $f-g=0$ and hence $f=g$. 
\end{proof}

\begin{lemma} \label{Anupam1}
Let $M$ be any right $R$-module and $L_1, L_2$ be two small submodules of $M$. Let $\sigma: \frac{M}{L_1} \rightarrow \frac{M}{L_2}$ be an epimorphism and $\eta: M\rightarrow M$ be a lifting of $\sigma$. Then
\begin{enumerate}[(i)]
\item $\eta$ is an epimorphism.
\item If $\sigma$ is an isomorphism and $\Ker(\eta)$ is a summand of $M$, then $\eta$ is an automorphism.
\end{enumerate}
\end{lemma}

\begin{proof}
\begin{enumerate}[(i)]
\item The hypothesis gives $\eta(M)+L_2=M$. Since $L_2\subset_{s} M$, we have $\eta(M)=M$. Hence $\eta$ is an epimorphism.

\item The hypothesis gives that $\Ker(\eta) \subseteq L_1$. Therefore $\Ker(\eta)\subset_{s} M$. Also, by the hypothesis, $\Ker(\eta)$ is a summand of $M$. Thus $\Ker(\eta)=0$ and hence $\eta$ is an automorphism.

\end{enumerate}
\end{proof}

\noindent Now we are ready to prove the following 

\begin{theorem} \label{char}
Let $P$ be a projective module and $K\subset_{s}P$. Then $M=\frac{P}{K}$ is dual automorphism-invariant if and only if $\sigma(K)=K$ for any automorphism $\sigma$ of $P$. \end{theorem}

\begin{proof}
Let $M=\frac{P}{K}$ be a dual automorphism invariant module. Let $\sigma:P\rightarrow P$ be an automorphism. The map $\sigma$ induces an epimorphism $\bar{\sigma}: \frac{P}{K} \rightarrow \frac{P}{K+\sigma(K)}$ given by $\bar{\sigma}(x+K)=\sigma(x)+K+\sigma(K)$. As $\Ker(\bar{\sigma})=\frac{\sigma^{-1}(K)+K}{K}$ is small in $M=P/K$, $\bar{\sigma}$ lifts to an automorphism $\eta$ of $M$ and $\eta^{-1}(\frac{K+\sigma(K)}{K})=\frac{\sigma^{-1}(K)+K}{K}$. Now $\eta$ lifts to an endomorphism $\lambda$ of $P$. By Lemma \ref{Anupam1}, $\lambda$ is an automorphism of $P$. Then $\lambda(K) \subseteq K$. If $K\subsetneq \lambda^{-1}(K)$, the mapping $\eta$ which is induced by $\lambda$ cannot be an automorphism. Hence $\lambda(K)=K$. As $\eta(\frac{\sigma^{-1}(K)+K}{K})=\frac{K+\sigma(K)}{K}$, we get $\lambda(\sigma^{-1}(K)+K)=K+\sigma(K)$. Let $C=\sigma^{-1}(K)$. Now $C\subset_{s} P$.     
We have two mappings  $\bar{\lambda}$ and $\bar{\mu}$ given as follows;
\[\bar{\lambda}:P\rightarrow P/K\] such that $\bar{\lambda}(x)=\lambda(x)+K$, and
\[\bar{\mu}:P\rightarrow P/K\] such that $\bar{\mu}(x)=\sigma(x)+K$.

Clearly $\bar{\mu}(C)=0$. Now $\eta(\frac{\sigma^{-1}(K)+K}{K})=\frac{\lambda(\sigma^{-1}(K))+\lambda(K)}{K}=\frac{\lambda(\sigma^{-1}(K))+K}{K}=\frac{\lambda(C)+K}{K}$. Hence $\bar{\lambda}(C)=\frac{\sigma(K)+K}{K}$. For $\bar{P}=\frac{P}{K}$, we can take $\frac{\bar{P}}{\bar{\lambda}(C)}=\frac{P}{\sigma(K)+K}$.  
Let $\pi:\frac{P}{K}\rightarrow \frac{P}{K+\sigma(K)}$ be a natural mapping. Set $\bar{\lambda}'=\pi \bar{\lambda}$, $\bar{\mu}'=\pi \bar{\mu}$. Let $x\in P$. Then $\bar{\lambda'}(x)=\pi(\lambda(x)+K)=\pi \eta (x+K)$. Now $\eta(x+K)=y+K$ for some $y\in P$. Thus $\bar{\lambda'}(x)=y+\sigma(K)+K=\bar{\sigma}(x+K)=\sigma(x)+\sigma(K)+K=\bar{\mu'}(x)$. Hence $\bar{\lambda'}=\bar{\mu'}$.  By Lemma \ref{map1}, we conclude that $\bar{\lambda}=\bar{\mu}$. This gives $\bar{\mu}(K)=\bar{\lambda}(K)=\bar{0}$, as $\lambda(K)=K$, we get $\frac{\sigma(K)+K}{K}=\bar{0}$. Hence $\sigma(K)\subseteq K$. By considering $\sigma^{-1}$, we get $\sigma^{-1}(K)\subseteq K$, therefore $K\subseteq \sigma(K)$. Hence $\sigma(K)=K$. 

Conversely, let $\sigma(K)=K$ for any automorphism $\sigma$ of $P$. Let $\overline{L_1}=\frac{L_1}{K}, \overline{L_2}=\frac{L_2}{K}$ be two small submodules of $M$ and $\sigma: \frac{M}{\overline{L_1}}\rightarrow \frac{M}{\overline{L_2}}$ be an epimorphism with $\Ker(\sigma) \subset_{s} \frac{M}{\overline{L_1}}$. Now $\Ker(\sigma)=\frac{\overline{L}}{\overline{L_1}}$, where $L$ is some submodule of $P$ containing $K$. Then $\bar{L}\subset_{s} M$ and hence $L\subset_s P$. Now $\sigma$ induces an epimorphism $\sigma': \frac{P}{L_1}\rightarrow \frac{P}{L_2}$ such that for any $x\in P$, $\sigma'(x+L_1)=y+L_2$ if and only if $\sigma(\bar{x}+\overline{L_1})=\bar{y}+\overline{L_2}$. Now $\Ker(\sigma')=\frac{L}{L_1}\subset_{s} \frac{P}{L_1}$, and $\sigma'$ is an epimorphism. It lifts to an endomorphism $\eta$ of $P$. Then $\Ker(\eta)\subseteq L$, and therefore $\Ker(\eta)\subset_s P$. The above lemma gives that $\eta$ is an automorphism of $P$. By the hypothesis, $\eta(K)=K$. Hence $\eta$ induces an automorphism $\bar{\eta}:M\rightarrow M$. This $\bar{\eta}$ lifts $\sigma$. Hence $M$ is dual automorphism-invariant.    
\end{proof}

\noindent We have already seen that if $M$ is a supplemented dual automorphism-invariant module, then $M$ satisfies the property ($D3$). Since every module over a right perfect ring is supplemented, it follows that every dual automorphism-invariant module over a right perfect ring satisfies the property ($D3$).

Now, for a lifting module over a right perfect ring, we have the following

\begin{proposition} \label{discrete}
Let $R$ be a right perfect ring and let $M$ be a right $R$-module such that $M$ is lifting. If $M$ is a dual automorphism-invariant module, then $M$ is discrete.
\end{proposition}

\begin{proof}
Let $M$ be a dual automorphism-invariant lifting module. By Proposition \ref{d3}, $M$ satisfies the property $(D3)$. Thus $M$ is a quasi-discrete module with the property that every epimorphism $f\in \End(M)$ with small kernel is an isomorphism. Hence, by \cite[Lemma 5.1]{MM}, $M$ is a discrete module.
\end{proof}

\noindent Next, we proceed to establish some decomposition results for discrete modules. This will help us in the study of dual automorphism-invariant lifting modules over right perfect rings. 
 
\begin{lemma} \label{decomposition}
Let $R$ be a right perfect ring and let $M=P/K$ be a right $R$-module where $P$ is projective and $K\subset_{s}P$. Suppose $M$ is a discrete module. Then
\begin{enumerate}[(i)]
\item If $P$ decomposes as $P=P_1\oplus P_2$, then we get $M=M_1\oplus M_2$ with 
\[M_1=\frac{P_1+K}{K}, M_2=\frac{P_2+K}{K};\]
and $K=K_1\oplus K_2$ with
\[K_1=K\cap P_1, K_2=K\cap P_2.\]
\noindent This shows any decomposition of $P$ gives rise to natural decompositions of both $M$ and $K$.
\item  If $\sigma \in End(P)$ is an idempotent, then $\sigma(K)\subseteq K$.
\end{enumerate}
\end{lemma}

\begin{proof}
\begin{enumerate}[(i)]
\item Let $P=P_1\oplus P_2$. Then $M=\frac{P_1+K}{K} + \frac{P_2+K}{K}$. Let $L_1$ and $L_2$ be projections of $K$ in $P_1$ and $P_2$ respectively. Then $L_1, L_2, L_1+L_2$ are small in $P$. Now $P_1 \cap (P_2+K)\subseteq L_1+L_2$, so $(P_1+K)\cap (P_2+K)\subseteq K+L_1+L_2$. This gives that $(\frac{P_1+K}{K}) \cap (\frac{P_2+K}{K})\subset_{s} M$. Since $M$ satisfies the property $(D1)$, we get that $M=\frac{A}{K} + \frac{B}{K}$ such that $\frac{A}{K}, \frac{B}{K}$ are summands of $M$ contained in $\frac{P_1+K}{K}$, $\frac{P_2+K}{K}$ respectively and are supplements of $\frac{P_2+K}{K}$, $\frac{P_1+K}{K}$ respectively. As $M$ satisfies the property $(D3)$, $\frac{A}{K} \cap \frac{B}{K}$ is a summand of $M$. However, $\frac{A}{K} \cap \frac{B}{K} \subseteq (\frac{P_1+K}{K}) \cap (\frac{P_2+K}{K})$ gives that $\frac{A}{K} \cap \frac{B}{K}$ is small in $M$. Therefore $M=\frac{A}{K} \oplus \frac{B}{K}$ and hence        
\[M=\frac{P_1+K}{K}\oplus \frac{P_2+K}{K}.\]
Let $K_1=K\cap P_1, K_2=K\cap P_2$. We have an isomorphism $\varphi:\frac{P_1}{K_1}\oplus \frac{P_2}{K_2}\rightarrow M$ given by $\varphi(x_1+K_1, x_2+K_2)=x_1+x_2+K$ where $x_1\in P_1, x_2\in P_2$. As $\varphi(\frac{K}{K_1+K_2})=0$, we get $K=K_1\oplus K_2$. Hence $K=(K\cap P_1) \oplus (K\cap P_2)$.
\item Let $P_1=\sigma P$, and $P_2=(1-\sigma)P$. Then $P=P_1\oplus P_2$. By (i), we have $K=K_1\oplus K_2$ where $K_1=K\cap P_1, K_2=K\cap P_2$. Clearly then $\sigma(K)\subseteq K$.
\end{enumerate}
\end{proof}

\noindent A ring $R$ is called a {\it clean ring} if each element $a\in R$ can be expressed as $a=e+u$, where $e$ is an idempotent in $R$ and $u$ is a unit in $R$. A module $M$ is called a {\it clean module} if $End(M)$ is a clean ring. The class of clean modules includes continuous modules, discrete modules, flat cotorsion modules, and quasi-projective right modules over a right perfect ring.

In the next theorem we show that every dual automorphism-invariant lifting module over a right perfect ring is quasi-projective.

\begin{theorem} \label{qp}
Let $R$ be a right perfect ring and let $M$ be a lifting right $R$-module. Then $M$ is dual automorphism-invariant if and only if $M$ is quasi-projective. 
\end{theorem}

\begin{proof}
Suppose $M$ is dual automorphism-invariant. Since $M$ has a projective cover, we set $M=P/K$, where $P$ is projective and $K\subset_{s}P$. Let $\sigma \in End(P)$. We know that $End(P)$ is clean (see \cite{CKLNZ}). Therefore, $\sigma=\alpha + \beta$ where $\alpha$ is an idempotent in $End(P)$ and $\beta$ is an automorphism on $P$. Since $M$ is a dual automorphism-invariant lifting module over a right perfect ring, by Proposition \ref{discrete}, $M$ is discrete. Therefore, by Lemma \ref{decomposition}(ii), $\alpha(K) \subseteq K$. Since $M$ is a dual automorphism-invariant module, by Theorem \ref{char}, $\beta(K) \subseteq K$. Thus $\sigma(K) =(\alpha+\beta)(K)\subseteq K$. Hence $M$ is quasi-projective. The converse follows from Proposition \ref{pseudo}.  
\end{proof}

\begin{theorem} \label{relqp}
Let $R$ be a right perfect ring. If $M=M_1\oplus M_2$ is a dual automorphism-invariant right $R$-module, then both $M_1$ and $M_2$ are dual automorphism-invariant and they are projective relative to each other.
\end{theorem}

\begin{proof}
We have already seen that a direct summand of a dual automorphism-invariant module is dual automorphism-invariant. 


Now, we proceed to show that $M_1$ and $M_2$ are projective relative to each other. Let $M_1=P_1/K_1$ and $M_2=P_2/K_2$ where $P_1, P_2$ are projective and $K_1\subset_{s} P_1, K_2\subset_{s} P_2$. Then $M=M_1\oplus M_2=\frac{P_1\oplus P_2}{K_1\oplus K_2}$. Note that the decomposition $M=M_1\oplus M_2$ gives rise to decomposition $P=P_1\oplus P_2$, where $M_1=\frac{P_1+K}{K}$ and $M_2=\frac{P_2+K}{K}$ and $K=K_1\oplus K_2$ where $K_1=K\cap P_1$, $K_2=K\cap P_2$. Thus $M_1\cong P_1/K_1$ and $M_2\cong P_2/K_2$. 

Let $\overline{L_2}=L_2/K_2$ be any submodule of $M_2$. Consider the exact sequence $M_2\rightarrow M_2/\overline{L_2}\rightarrow 0$. Let $\lambda: M_1\rightarrow M_2/\overline{L_2}$ be a homomorphism. This gives us a mapping \[\lambda': \frac{P_1}{K_1}\rightarrow \frac{P_2}{L_2}\] with $\lambda'(x_1+K_1)=x_2 +L_2$ if $\lambda(x_1+K_1)=(x_2+K_2)+\frac{L_2}{K_2}$. 

It lifts to a homomorphism $\mu: P_1\rightarrow P_2$. Then $P=P'_1\oplus P_2$ where $P'_1=\{x_1+\lambda'(x_1): x_1\in P_1\}$. We get an automorphism \[\sigma: P\rightarrow P\] where $\sigma(x_1+x_2)=x_1+\lambda'(x_1)+x_2$. 

Since $M$ is dual automorphism-invariant, we have $\sigma(K)=K=\sigma(K_1)\oplus \sigma(K_2)=K\cap P'_1 \oplus K\cap P_2$. This gives a decomposition \[M=\frac{P'_1+K}{K} \oplus \frac{P_2+K}{K}.\] We have an isomorphism \[\sigma': \frac{P_1+K}{K}\rightarrow \frac{P'_1+K}{K}\] given by $\sigma'(x_1+K)=\sigma(x_1)+K=x_1+\lambda'(x_1)+K$. Now, if $x_1\in K$, then $x_1+\lambda'(x_1)\in K$. This gives $\lambda'(x_1)\in K\cap P_2=K_2$. Hence $\lambda'$ induces mapping \[\bar{\mu}: \frac{P_1}{K_1}\rightarrow \frac{P_2}{K_2}\] given by $\bar{\mu}(x+K_1)=\lambda'(x)+K_2$. This shows that $M_1$ is projective with respect to $M_2$. Similarly, it can be shown that $M_2$ is projective with respect to $M_1$.   
\end{proof}

As a consequence it follows that

\begin{corollary}
If $R$ is a right perfect ring, then a right $R$-module $M$ is quasi-projective if and only if $M\oplus M$ is dual automorphism-invariant.
\end{corollary}

\begin{proof}
Let $R$ be a right perfect ring. Suppose $M$ is a quasi-projective right $R$-module. Then $M\oplus M$ is quasi-projective and hence dual automorphism-invariant. Conversely, suppose $M\oplus M$ is dual automorphism-invariant. Then by Theorem \ref{relqp}, $M$ is $M$-projective, that is, $M$ is quasi-projective.
\end{proof}

\begin{proposition}
Let $R$ be an artinian serial ring. Then a right $R$-module $M$ is dual automorphism-invariant if and only if $M$ is quasi-projective.
\end{proposition}

\begin{proof}
Suppose $M$ is dual automorphism-invariant. Since $R$ is artinian serial, $M=\oplus_{i=1}^{n} M_i$, where each $M_i$ is uniserial. Since $M$ is dual automorphism-invariant, by Theorem \ref{relqp}, each $M_i$ is projective with respect to $M_j$, for each $j\neq i$. 

Let $M_i, M_j$ be such that $\frac{M_i}{M_iJ(R)}\cong \frac{M_j}{M_jJ(R)}$. Then, since $M_i, M_j$ are projective relative to each other, we can lift this isomorphism to give $M_i\cong M_j$. 

So now $M=\oplus_{i=1}^{m} L_i$, where $L_i=\oplus_{k\in \Lambda} M_k$ with $M_i\cong M_k$ for each $i, k \in \Lambda$. Let $t$ be the length of $M_k\subset L_i$. Then, as $R$ is an artinian serial ring, $M_i$ is projective as an $R/J^t(R)$-module for each $i\in \Lambda$. This shows that $L_i$ is $M$-projective. Consequently, it follows that $M$ is $M$-projective. Thus $M$ is quasi-projective. The converse is obvious.     
\end{proof}

\bigskip

\section{Problems}

\bigskip

\begin{problem}
Let $M_1$ and $M_2$ be dual automorphism-invariant modules such that $M_1$ is $M_2$-projective and $M_2$ is $M_1$-projective. Is $M_1\oplus M_2$ a dual automorphism-invarint module?
\end{problem}

\begin{problem}
Characterize von Neumann regular rings over which every right $R$-module is dual automorphism-invariant.
\end{problem}

\begin{problem}
Characterize rings over which each cyclic module is dual automorphism-invariant?
\end{problem}

\bigskip

\begin{center}
Acknowledgment
\end{center} 

The first author would like to thank the Saint Louis University for its warm hospitality during his visit in September 2011.

\bigskip

\bigskip

\bigskip


\bigskip

\bigskip

\begin{thebibliography}{99}

\bibitem{Bass}
H. Bass, Finitistic dimension and a homological generalization of semi-primary rings, Trans. Amer. Math. Soc. 95 (1960), 466-488.

\bibitem{CKLNZ}
V. P. Camillo, D. Khurana, T. Y. Lam, W. K. Nicholson, Y. Zhou, Continuous modules are clean, J. Algebra, 304 (2006), 94-111.

\bibitem{CLVW}
J. Clark, C. Lomp, N. Vanaja, R. Wisbauer, Lifting Modules: Supplements and projectivity in module theory, Frontiers in Mathematics, Birkhauser Verlag, Basel, 2006.

\bibitem{Fuchs}
L. Fuchs, Infinite Abelian Groups, Volume 1, Academic Press, 1970.

\bibitem{FR}
L. Fuchs, K. M. Rangaswamy, Quasi-projective abelian groups, Bull. Math. Soc. France, 98 (1970), 5-8.

\bibitem{ZL}
T. K. Lee, Y. Zhou, Modules which are invariant under automorphisms of their injective hulls, J. Alg. Appl., to appear.

\bibitem{MV}
G. O. Michler and O. E. Villamayor, On Rings whose Simple Modules are Injective, Journal of Algebra 25 (1973), 185-201.

\bibitem{MM}
S. H. Mohamed, B. J. Muller, Continuous and Discrete Modules, London Math. Soc. Lecture Note Ser., vol. 147, Cambridge Univ. Press, Cambridge, 1990.

\bibitem{MS}
S. H. Mohamed, S. Singh, Generalizations of decomposition theorems known over perfect rings, J. Aust. Math. Soc. 24 (1977), 496-510.

\bibitem{SS} 
S. Singh, A. K. Srivastava,  Rings of invariant module type and automorphism-invariant modules , Contemporary Mathematics, Amer. Math. Soc., to appear.


\bibitem{Teply}
M. L. Teply, Pseudo-injective modules which are not quasi-injective, Proc. Amer. Math. Soc., vol. 49, no. 2 (1975), 305-310.

\bibitem{Z}
J. Zelmanowitz, Regular modules, Trans. Amer. Math. Soc. 163 (1972), 341-355.

\end{thebibliography}
\end{document}